\documentclass[11pt, one side,emlines]{amsart}
\usepackage{amssymb,latexsym,xy,eucal,mathrsfs, graphicx, tikz, amsmath, caption}
\textwidth=15.2cm \textheight=24cm \theoremstyle{plain}
\newtheorem{lemma}{Lemma}[section]
\newtheorem{theorem}[lemma]{Theorem}

\theoremstyle{definition}

\def\prfname{{\it Proof of Theorem }}

\numberwithin{equation}{section}  \voffset
-55truept \hoffset -15truept
\begin{document}
	\baselineskip 15truept
	\title{Constructing edge-disjoint spanning trees in augmented cubes}
	\maketitle 
	\begin{center} 
		
		S. A. Mane\\  
		{\small Center for Advanced Studies in Mathematics,
			Department of Mathematics,\\ Savitribai Phule Pune University, Pune-411007, India.}\\
		\email{\emph{manesmruti@yahoo.com}} 
	\end{center} 
	\begin{abstract}
		Let $T_1, T_2,.......,T_k$ be spanning trees in a graph $G$. 
If for any pair of vertices $\{u,v\}$ of $G$, 
the paths between $u$ and $v$ in every $T_i$( $1\leq i\leq k$)
do not contain common edges then $T_1, T_2,.......,T_k$ 
are called edge-disjoint spanning trees in $G$. The design of multiple 
edge-disjoint spanning trees has applications
to the reliable communication protocols.
The $n-$dimensional augmented cube, denoted as $AQ_n$,
 a variation of the hypercube, possesses some properties superior to those of the hypercube.
  For $AQ_n$ ($n \geq 3$), construction of $n-1$ edge-disjoint spanning trees is given the result is
  optimal with respect to the number of edge-disjoint spanning trees.
\end{abstract}

\section{Introduction} 
     \indent A graph $G$ is a triple consisting of a vertex set $V(G)$, an edge set $E(G)$, and a
     relation that associates with each edge two vertices called its endpoints\cite{we}.
     The topology of the
     network is modeled in the form of a graph whose vertices correspond to nodes, while edges represent direct physical
     connections between nodes. This paper deals with the well-established problem of handling the maximum possible
     number of communication requests without using a single physical link more than once, known as the
     edge-disjoint spanning trees Problem.
     The hypercubes $Q_n$ are one of the most versatile and efficient interconnection
     networks discovered for parallel computation. Many variants of the hypercube
     have been proposed. The augmented cube, proposed by Choudam and Sunitha\cite{c}, 
     is one of such variations. An $n-$dimensional augmented cube $AQ_n$ can be formed as an
     extension of $Q_n$ by adding some links. For any positive integer $n$, $AQ_n$ is a vertex transitive, 
     $(2n-1)$-regular and  $(2n-1)$-connected graph with $2^n$ vertices. $AQ_n$ retains all favorable
     properties of $Q_n$ since $Q_n \subset AQ_n$. Moreover, $AQ_n$ possesses some embedding
     properties that $Q_n$ does not. The main merit of augmented cubes is that their diameters are 
     about half of those of the corresponding hypercubes. \\
     A tree $T$ is called a spanning tree of a graph $G$ if $V(T)= V(G)$.
     Two spanning trees $T_1$ and $T_2$ in $G$ are edge-disjoint if $E(T_1) \cap E(T_2)= \phi$.  \\
      The edge-disjoint spanning trees(EDSTs for short) problem has received a great deal of attention in 
      recent years because of its numerous applications on interconnection networks 
     such as fault-tolerant broad casting and secure message distribution\cite{f,j,h,t,l,n,w,x,y}.\\ 
    Barden, Hadas, Davis and Williams\cite{b} proved there exist $n$ EDSTs  
      in a hypercube of dimension $2n$ and provided a construction for obtaining the maximum number of EDSTs in a hyperube. 
    Wang, Shen and Fan\cite{wa} proved existence of $n-1$ EDSTs in an augmented cube $AQ_n$ and they asked,
    ``how to derive an effective algorithm that constructs edge-disjoint spanning trees based on our algorithm,
    in an augmented cube?'' 
     Motivated by this question, we provide the construction for obtaining the maximum number of EDSTs ($n-1$ EDSTs )
     in augmented cube $AQ_n$ ( $n \geq 3$ ).

        \section {Preliminaries}
        
 The definition of the $n-$dimensional augmented cube is stated as the following. Let $n \geq 1$ be 
 an integer. The $n$-dimensional augmented cube, denoted by $AQ_n$, is a graph with $2^n$ vertices, and each
  vertex $u$ can be distinctly labeled by an $n$-bit binary string, $u = u_1u_2....u_n$. $AQ_1$ is the graph $K_2$ with 
  vertex set $\{0, 1\}$. For
  $n \geq 2$, $AQ_n$ can be recursively constructed by two copies of $AQ_{n-1}$, denoted by $AQ^0_{n-1}$
  and $AQ^1_{n-1}$, and by adding
  $2^n$ edges between $AQ^0_{n-1}$ and $AQ^1_{n-1}$
  as follows:\\ Let $V(AQ^0_{n-1}) = \{0u_2....u_n : u_i \in \{0, 1\}, 2 \leq i \leq n\}$ and 
  $V(AQ^1_{n-1}) = \{1v_2....v_n : v_i \in \{0, 1\}, 2 \leq i \leq n\}$. A vertex $u = 0u_2....u_n$ of $AQ^0_{n-1}$
  is joined to a vertex $v = 1v_2....v_n $ of $AQ^1_{n-1}$
  if and only if for every $i$, $2 \leq i \leq n$ either\\
  1. $u_i = v_i$; in this case an edge $\langle u, v \rangle$ is called a hypercube edge and we say $v = u^h$, or\\
  2. $u_i = \overline{v_i}$; in this case an edge $\langle u, v \rangle$ is called a complement edge and we say $v = u^c$.\\
  Let $E^h_n = \{\langle u, u^h \rangle : u \in V(AQ^0_{n-1})\}$ and
  $E^c_n = \{\langle u, u^c \rangle : u \in V(AQ^0_{n-1})\}$. See Fig.$1$.\\
  \begin{center}
  	\unitlength 1mm 
  	\linethickness{0.4pt}
  	\ifx\plotpoint\undefined\newsavebox{\plotpoint}\fi 
  	\begin{picture}(294.151,50.401)(0,0)
  	\put(15,22.5){\line(0,1){18.25}}
  	\put(34.5,22.5){\line(0,1){18.25}}
  	\put(73.25,22.5){\line(0,1){18.25}}
  	\put(121.5,22.5){\line(0,1){18.25}}
  	\put(57.25,22.5){\line(0,1){18.25}}
  	\put(96,22.5){\line(0,1){18.25}}
  	\put(144.25,22.5){\line(0,1){18.25}}
  	\put(34.75,40.5){\line(1,0){22.5}}
  	\put(73.5,40.5){\line(1,0){22.5}}
  	\put(121.75,40.5){\line(1,0){22.5}}
  	\put(34.75,23){\line(1,0){22.25}}
  	\put(73.5,23){\line(1,0){22.25}}
  	\put(121.75,23){\line(1,0){22.25}}
  	\multiput(34.5,40.25)(.0439453125,-.0336914063){512}{\line(1,0){.0439453125}}
  	\multiput(73.25,40.25)(.0439453125,-.0336914063){512}{\line(1,0){.0439453125}}
  	\multiput(121.5,40.25)(.0439453125,-.0336914063){512}{\line(1,0){.0439453125}}
  	\multiput(34.75,23.25)(.0439453125,.0336914063){512}{\line(1,0){.0439453125}}
  	\multiput(73.5,23.25)(.0439453125,.0336914063){512}{\line(1,0){.0439453125}}
  	\multiput(121.75,23.25)(.0439453125,.0336914063){512}{\line(1,0){.0439453125}}
  	\put(96.25,40.25){\line(1,0){25.5}}
  	\put(96,23.5){\line(1,0){25.5}}
  	\multiput(73.5,40.25)(.097082495,-.0337022133){497}{\line(1,0){.097082495}}
  	\multiput(96,23)(.0929672447,.0337186898){519}{\line(1,0){.0929672447}}
  	\multiput(96,40.25)(.0942460317,-.0337301587){504}{\line(1,0){.0942460317}}
  	\multiput(121.5,40.25)(-.0920303605,-.0336812144){527}{\line(-1,0){.0920303605}}
  	\qbezier(73.25,40.5)(111,55.5)(143.75,40.5)
  	\qbezier(73.5,22.75)(105.375,8.625)(143.75,23)
  	\put(293.25,123.5){\circle*{1.803}}
  	\put(14.75,40.75){\circle*{1.581}}
  	\put(15,21.75){\circle*{1.581}}
  	\put(34.25,23){\circle*{1.581}}
  	\put(34.5,39.75){\circle*{1.581}}
  	\put(57.25,39.75){\circle*{1.581}}
  	\put(57.25,22.5){\circle*{1.581}}
  	\put(74,22.5){\circle*{1.581}}
  	\put(73.25,40.5){\circle*{1.581}}
  	\put(96.5,40.5){\circle*{1.581}}
  	\put(122,40.25){\circle*{1.581}}
  	\put(144.5,40){\circle*{1.581}}
  	\put(144,23.25){\circle*{1.581}}
  	\put(121.75,22.75){\circle*{1.581}}
  	\put(96,23.5){\circle*{1.581}}
  	\put(14.75,44.75){\makebox(0,0)[cc]{$\tiny{0}$}}
  	\put(14.75,17.25){\makebox(0,0)[cc]{$\tiny{1}$}}
  	\put(33,44.75){\makebox(0,0)[cc]{$\tiny{00}$}}
  	\put(57.75,46.5){\makebox(0,0)[cc]{$\tiny{10}$}}
  	\put(32.5,18.5){\makebox(0,0)[cc]{$\tiny{01}$}}
  	\put(57.75,18){\makebox(0,0)[cc]{$\tiny{11}$}}
  	\put(74,45.75){\makebox(0,0)[cc]{$\tiny{000}$}}
  	\put(73,18.75){\makebox(0,0)[cc]{$\tiny{010}$}}
  	\put(95,43.75){\makebox(0,0)[cc]{ $\tiny{001}$}}
  	\put(95,18.75){\makebox(0,0)[cc]{$\tiny{011}$}}
  	\put(121.5,44.25){\makebox(0,0)[cc]{$\tiny{101}$}}
  	\put(145.75,44.5){\makebox(0,0)[cc]{$\tiny{100}$}}
  	\put(145.75,18.5){\makebox(0,0)[cc]{$\tiny{110}$}}
  	\put(121.25,19.5){\makebox(0,0)[cc]{$\tiny{111}$}}
  	\put(14.75,13){\makebox(0,0)[cc]{$\tiny AQ_{1}$}}
  	\put(44.75,14){\makebox(0,0)[cc]{$\tiny AQ_{2}$}}
  	\put(106.25,11.75){\makebox(0,0)[cc]{$\tiny AQ_{3}$}}
  	\put(64,6.5){\makebox(0,0)[cc]{Fig.$1$}}
  	\put(59,31.25){\makebox(0,0)[cc]{$\tiny AQ^1_{1}$}}
  	\put(31,32.25){\makebox(0,0)[cc]{$\tiny AQ^0_{1}$}}
  	\put(77,31.5){\makebox(0,0)[cc]{$\tiny AQ^0_{2}$}}
  	\put(140.25,31.75){\makebox(0,0)[cc]{$\tiny AQ^1_{2}$}}
  	\end{picture}
  	
  \end{center}
  
   For undefined terminology and
   notations see \cite{we}.

     \section {Construction of edge-disjoint spanning trees in augmented cubes }
     
     Our proof is by induction. \\
     As $AQ_n$ is $(2n-1)$-regular, $ |E(AQ_n)|=(2n-1)(2^{n-1})= n2^n-2^{n-1}$. When we construct any spanning tree on $2^n$ vertices of $AQ_n$ we need exactly $2^n -1$ edges hence we can construct at most $n-1$ EDSTs. Still, for $n \geq 3$, $2^{n-1}+ n-1 (< 2^n -1)$ number of edges remain uncovered by these $n-1$ EDSTs, but by our method, we are able to construct the tree on $2^{n-1}+ n$ vertices containing uncovered edges.

 \begin{theorem}Let $n\geq 3$ be an integer. There exist $n-1$ edge-disjoint spanning trees in augmented cube $AQ_n$. 
 \end{theorem} 
     
\begin{proof} First if $n = 3$, we construct two EDSTs $T_1$ and $T_2$ as follows. See Fig.2.\\
	\begin{center}
\unitlength 0.5mm 
\linethickness{0.4pt}
\ifx\plotpoint\undefined\newsavebox{\plotpoint}\fi 
\begin{picture}(344.75,88.375)(0,0)
\put(9.291,21.041){\circle*{1.581}}
\put(112.041,21.041){\circle*{1.581}}
\put(45.791,21.041){\circle*{1.581}}
\put(149.291,21.041){\circle*{1.581}}
\put(9.791,53.54){\circle*{1.581}}
\put(113.291,53.54){\circle*{1.581}}
\put(47.041,53.54){\circle*{1.581}}
\put(150.541,53.54){\circle*{1.581}}
\put(4.25,57.75){\makebox(0,0)[cc]{$\tiny{000}$}}
\put(53,56.5){\makebox(0,0)[cc]{ $\tiny{001}$}}
\put(-1,14.75){\makebox(0,0)[cc]{$\tiny{010}$}}
\put(48,14.25){\makebox(0,0)[cc]{$\tiny{011}$}}
\put(159,16.75){\makebox(0,0)[cc]{$\tiny{111}$}}
\put(157.25,58.5){\makebox(0,0)[cc]{$\tiny{101}$}}
\put(104.75,55){\makebox(0,0)[cc]{$\tiny{100}$}}
\put(105,14){\makebox(0,0)[cc]{$\tiny{110}$}}
\put(9.75,53.5){\line(-1,0){.25}}
\put(9.25,21){\line(0,1){33.5}}
\put(8.5,54.5){\line(0,1){0}}
\put(9.25,21.25){\line(1,0){4.25}}
\put(12.75,21.25){\line(1,0){33}}
\put(45.75,21.25){\line(0,1){0}}
\put(45.75,21.25){\line(0,1){0}}
\qbezier(8.75,20.5)(9,22.875)(9.25,21.75)
\qbezier(9.25,21.75)(9,22)(8.75,21.25)
\qbezier(8,21.25)(74,42.25)(112,21.25)
\multiput(112,21.25)(.03941908714,.03371369295){964}{\line(1,0){.03941908714}}
\qbezier(45.75,21)(111.875,50.625)(149.5,20.75)
\qbezier(112.75,53.25)(65.625,88.375)(10,54)
\qbezier(150.5,54.25)(101.75,86.125)(47,54.5)
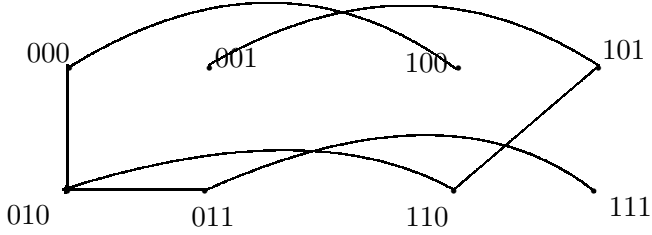
\put(75.25,1.75){\makebox(0,0)[cc]{Fig.2(a). Spanning tree $T_1$ in $AQ_3$}}
\put(194.291,22.041){\circle*{1.581}}
\put(298.541,20.041){\circle*{1.581}}
\put(231.541,22.041){\circle*{1.581}}
\put(335.041,22.041){\circle*{1.581}}
\put(195.541,54.54){\circle*{1.581}}
\put(299.041,54.54){\circle*{1.581}}
\put(232.791,54.54){\circle*{1.581}}
\put(336.291,54.54){\circle*{1.581}}
\put(190,58.75){\makebox(0,0)[cc]{$\tiny{000}$}}
\put(238.75,57.5){\makebox(0,0)[cc]{ $\tiny{001}$}}
\put(184.75,15.75){\makebox(0,0)[cc]{$\tiny{010}$}}
\put(233.75,15.25){\makebox(0,0)[cc]{$\tiny{011}$}}
\put(344.75,17.75){\makebox(0,0)[cc]{$\tiny{111}$}}
\put(343,59.5){\makebox(0,0)[cc]{$\tiny{101}$}}
\put(290.5,56){\makebox(0,0)[cc]{$\tiny{100}$}}
\put(290.75,15){\makebox(0,0)[cc]{$\tiny{110}$}}
\put(195.5,54.5){\line(1,0){36.75}}
\put(232.25,54.5){\line(0,-1){33.25}}
\multiput(231.25,22.25)(.06920326864,.03370786517){979}{\line(1,0){.06920326864}}
\put(297.75,28){\line(0,-1){.25}}
\put(299,54.5){\line(-1,0){.25}}
\put(298.5,20.75){\line(0,1){33}}
\put(298.5,55){\line(1,0){37.75}}
\multiput(299,54.75)(.03765690377,-.03373430962){956}{\line(1,0){.03765690377}}
\multiput(194.25,22.25)(.046979866,.033557047){149}{\line(1,0){.046979866}}
\multiput(201.25,27.25)(.0396039604,.03372524752){808}{\line(1,0){.0396039604}}
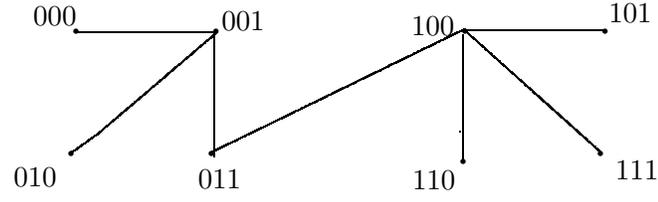
\put(263.25,4.25){\makebox(0,0)[cc]{Fig.2(b). Spanning tree $T_2$ in $AQ_3$}}
\end{picture}
	\end{center}

Edges uncovered by $T_1$ and $T_2$ again form a tree say $T_3$ on $7$ vertices. See Fig.3.\\
\begin{center}
\unitlength 0.6mm 
\linethickness{0.4pt}
\ifx\plotpoint\undefined\newsavebox{\plotpoint}\fi 
\begin{picture}(185.5,58.25)(0,0)
\put(35.041,20.791){\circle*{1.581}}
\put(138.541,20.791){\circle*{1.581}}
\put(72.291,20.791){\circle*{1.581}}
\put(175.791,20.791){\circle*{1.581}}
\put(36.291,53.29){\circle*{1.581}}
\put(139.791,53.29){\circle*{1.581}}
\put(175.291,53.54){\circle*{1.581}}
\put(73.541,53.29){\circle*{1.581}}
\put(30.75,57.5){\makebox(0,0)[cc]{$\tiny{000}$}}
\put(79.5,56.25){\makebox(0,0)[cc]{ $\tiny{001}$}}
\put(25.5,14.5){\makebox(0,0)[cc]{$\tiny{010}$}}
\put(74.5,14){\makebox(0,0)[cc]{$\tiny{011}$}}
\put(185.5,16.5){\makebox(0,0)[cc]{$\tiny{111}$}}
\put(183.75,58.25){\makebox(0,0)[cc]{$\tiny{101}$}}
\put(131.25,54.75){\makebox(0,0)[cc]{$\tiny{100}$}}
\put(131.5,13.75){\makebox(0,0)[cc]{$\tiny{110}$}}
\multiput(35.5,53.5)(.0343406593,-.0336538462){364}{\line(1,0){.0343406593}}
\multiput(48,41.25)(.0394308943,-.0337398374){615}{\line(1,0){.0394308943}}
\put(72.25,20.5){\line(0,1){0}}
\multiput(72.25,20.5)(-.03125,.15625){8}{\line(0,1){.15625}}
\multiput(73.5,53.25)(.06716804979,-.03371369295){964}{\line(1,0){.06716804979}}
\put(138.25,21.25){\line(1,0){37.5}}
\put(175.5,21.5){\line(0,1){14.25}}
\put(175.5,35.75){\line(0,1){18}}
\put(175.5,53.75){\line(0,1){0}}
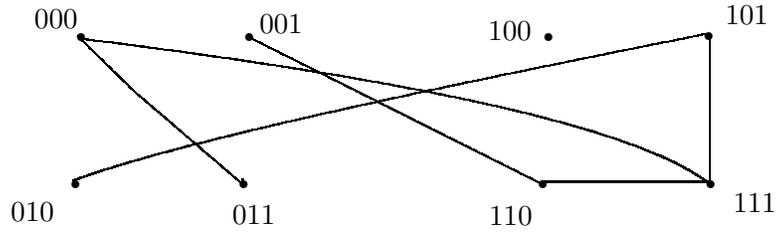
\put(103,1.25){\makebox(0,0)[cc]{Fig.3. Tree $T_3$ in $AQ_3$}}
\qbezier(34.5,21.5)(61.375,31.75)(174.75,54)
\qbezier(35.75,53)(152.875,38.625)(175.5,20.75)
\end{picture}
\end{center}

Let $AQ_{n+1}$ be decomposed into two augmented cubes say $AQ^0_n$ and $AQ^1_n$ with
vertex set say $ \{u^0_i, v^0_i : 1 \leq i \leq 2^{n-1}\}$ and  $\{u^1_i, v^1_i : 1 \leq i \leq 2^{n-1}\}$ respectively. 
Denote by $T^0_1, T^0_2,.......,T^0_{(n-1)}$  the EDSTs in $AQ^0_n$. Let the identical corresponding EDSTs
in $AQ^1_n$ be denoted by $T^1_1, T^1_2,.......,T^1_{(n-1)}$. Let the vertices of the tree say 
$T^0_n$ which is made up of uncovered edges of $n-1$ EDSTs in $AQ^0_n$ be denoted 
by $u^0_1, u^0_2,....u^0_{2^{n-1}}, v^0_1, v^0_2,....v^0_n$ and the identical corresponding 
vertices of the corresponding tree say $T^1_n$ in $AQ^1_n$ be denoted by
$u^1_1, u^1_2,....u^1_{2^{n-1}}, v^1_1, v^1_2,....v^1_n$.

We want to construct $n$ EDSTs in $AQ_{n+1}$, of which first $n-2$ EDSTs are constructed from 
$T^0_i$ and $T^1_i$ by adding a hypercube edge $\langle v^0_i, v^1_i \rangle \in E(AQ_{n+1})$ to 
connect two internal vertices $v^0_i \in V(T^0_i)$ and $v^1_i \in V(T^1_i)$, for $1\leq i \leq n-2$. See Fig.4.\\
\begin{center}
\unitlength 0.6mm 
\linethickness{0.4pt}
\ifx\plotpoint\undefined\newsavebox{\plotpoint}\fi 

\end{center}

The $(n-1)^{th}$ EDST is constructed from $T^0_{(n-1)}$ by adding all complement edges
$E^c_{n+1} = \{\langle u, u^c \rangle : u \in V(AQ^0_{n})\}$.See Fig.5.\\
\begin{center}
\unitlength 0.6mm 
\linethickness{0.4pt}
\ifx\plotpoint\undefined\newsavebox{\plotpoint}\fi 
%
\end{center}

The $n^{th}$ EDST is constructed from $T^1_{(n-1)}$ by adding to it  hypercube edges 
$\langle v^0_i, v^1_i \rangle \in E(AQ_{n+1}) $ to connect internal vertices
$v^0_i \in V(AQ^0_n)$ and $v^1_i \in V(T^1_{(n-1)})$ $n \leq i \leq 2^{n-1}$. 
The $2^{n-1}+ (n-1)$ vertices in $AQ^0_n$ which are not connected to $T^1_{(n-1)}$ via 
hypercube edges are $u^0_i$ ($1 \leq i \leq 2^{n-1}$) and $ v^0_1, v^0_2,....v^0_{n-1}$, 
but the edge $\langle v^0_{n}, v^1_{n} \rangle$ is included in $T^1_{(n-1)}$, by adding 
to $T^1_{(n-1)}$ the edges of the tree $T^0_n$ through the vertex $v^0_{n}$ connect all 
remaining vertices $ v^0_1, v^0_2,....v^0_{n-1}$ and $u^0_i$ ($1 \leq i \leq 2^{n-1}$) to $T^1_{(n-1)}$. See Fig.6. \\
\begin{center}
	\unitlength 0.6mm 
	\linethickness{0.4pt}
	\ifx\plotpoint\undefined\newsavebox{\plotpoint}\fi 

	
\end{center}

Still, we have uncovered edges of $AQ_{n+1}$ namely edges of tree $T^1_n$, the 
hypercube edge $\langle v^0_{n-1}, v^1_{n-1} \rangle$ and hypercube 
edges $\langle u^0_i, u^1_i \rangle$ ($1 \leq i \leq 2^{n-1}$). We can easily 
observe that these uncovered hypercube edges along with the tree $T^1_n$ again
form a new tree on $(2^{n-1}+ n) + (2^{n-1} +1) = 2^n + n+1$ vertices. See Fig.7. \\
\begin{center}
\unitlength 0.6mm 
\linethickness{0.4pt}
\ifx\plotpoint\undefined\newsavebox{\plotpoint}\fi 

           	
\end{center}

 \end{proof}

\noindent {\bf Acknowledgment:} The author gratefully
acknowledges the Department of Science and Technology, New Delhi, India
for the award of Women Scientist Scheme for research in Basic/Applied Sciences.

$$\diamondsuit\diamondsuit\diamondsuit$$

\end{document}